\documentclass[12pt,a4paper,reqno]{amsart}
\usepackage[centertags]{amsmath}
\usepackage{amsfonts}
\usepackage{amssymb}

\usepackage{amsthm}
\usepackage{newlfont}
\usepackage{color}
\usepackage[margin=1in]{geometry} 
\usepackage{graphicx}              
\usepackage{amsmath}               
\usepackage{amsfonts}              
\usepackage{amsthm}   
\newtheorem{theorem}{Theorem}[section]
\newtheorem{proposition}[theorem]{Proposition}

\newtheorem{corollary}[theorem]{Corollary}

\theoremstyle{definition}
\newtheorem{defn}[theorem]{Definition}
\newtheorem{remark}[theorem]{Remark}
\theoremstyle{example}

\theoremstyle{problem}

\theoremstyle{property}

\begin{document}

\title[Some inequalities for operator $(p,h)$-convex functions]{Some inequalities for operator $(p,h)$-convex functions}

\author[T. H. Dinh]{Trung Hoa Dinh$^*$}
\address{Dong A University,  33 Xo Viet Nghe Tinh, Da Nang, Vietnam
}
\email{trunghoa.math@gmail.com}

\author[Khue TB Vo]{Khue Thi Bich Vo$^*$}
\address{Quy Nhon University, Vietnam \\ and University of Finance - Marketing, 2/4 Tran Xuan Soan, Ho Chi Minh City, Vietnam}
\email{votbkhue@gmail.com}

\thanks{$^*$This research is funded by Vietnam National Foundation for Science and Technology Development (NAFOSTED) under  grant number 101.04-2014.40. This work was partially finished at Ton Duc Thang University, Vietnam.}

\begin{abstract}
Let $p$ be a positive number and $h$ a function on $\mathbb{R}^+$ satisfying $h(xy) \ge h(x) h(y)$ for any $x, y \in \mathbb{R}^+$. A non-negative  continuous function $f$ on $K (\subset \mathbb{R}^+)$ is said to be {\it operator $(p,h)$-convex} if 
\begin{equation*}\label{def}
f ([\alpha A^p + (1-\alpha)B^p]^{1/p}) \leq h(\alpha)f(A) +h(1-\alpha)f(B)
\end{equation*}
holds for all  positive semidefinite matrices $A, B$ of order $n$ with spectra in $K$, and for any $\alpha \in (0,1)$. 

In this paper, we study properties of operator $(p,h)$-convex functions and prove the Jensen, Hansen-Pedersen type inequalities for them. We also give some equivalent conditions for a function to become an operator $(p,h)$-convex. In applications, we obtain Choi-Davis-Jensen type inequality for operator $(p,h)$-convex functions and a relation between operator $(p,h)$-convex functions with operator monotone functions. 
\end{abstract}

\maketitle


\section{Introduction}
Let $\mathbb{M}_n$ be the space of $n\times n$ complex matrices, $M_n^+$ the positive part of  $\mathbb{M}_n$. Denote by $I_n$ and $O_n$ the identity and zero elements of $\mathbb{M}_n$, respectively. For self-adjoint matrices $A, B \in \mathbb{M}_n$ the notation $A \le B$ means that $B - A \in  \mathbb{M}_n^+$. The spectrum of a matrix $A \in \mathbb{M}_n$ is denoted by $\sigma(A)$. For a real-valued function $f$ of a real variable and a self-adjoint matrix $A \in \mathbb{M}_n$, the value $f(A)$ is understood by means of the functional calculus.

We assume further that $p$ is some positive number, $J$ is some interval in $\mathbb{R}^+$ such that $(0,1) \subset J$, and $K$ ($\subset \mathbb{R}^+$) is a $p$-convex subset of $\mathbb{R}^+$ (that means, $(\lambda x^p + (1-\lambda)y^p)^{1/p} \in K$ for all $x, y \in K$ and $\lambda \in [0,1]$). A function $h: J \rightarrow \mathbb{R}^+$ is called a {\it super-multiplicative} function if $h(xy) \geq h(x)h(y)$ for all $x, y \in J$. 

A non-negative function $f: K \rightarrow \mathbb{R}$ is said to be $h$-convex \cite{SV} if
for all $x,y \in K, \alpha \in (0,1)$ we have
$$
f(\alpha x+(1-\alpha)y) \le h(\alpha)f(x)+h(1-\alpha)f(y).
$$

In \cite{ZR} a more general  class of non-negative functions, the so-called $(p,h)$-convex functions is considered.

Let $h: J \rightarrow \mathbb{R}^+$ be a non-zero super-multiplicative function. A non-negative function $f: K \rightarrow  \mathbb{R}$ is said to be {\it $(p,h)$-convex} if 
\begin{equation}\label{def}
f \big([\alpha x^p + (1-\alpha)y^p]^{1/p}\big) \leq h(\alpha)f(x) +h(1-\alpha)f(y).
\end{equation}
for all $x, y \in K$ and $\alpha \in (0,1)$.

This class contains several well-known classes of functions such as non-negative convex functions, $h$- and $p$-convex functions, Godunova-Levin functions (or $Q$-class functions) and $P$-class functions that are considered by many authors. 

In \cite{Mo}-\cite{Mo4} M.S.Moslehian, M.Kian and others introduced operator $P$-class functions and operator $Q$-class functions. They studied properties and proved several inequalities for these classes of functions. 

Motivated by the above mentioned works, in this paper we define a class of operator functions which we call it operator $(p, h)$-convex. In Section 2 we study properties of operator $(p, h)$-convex functions. In Section 3 we prove some inequalities for operator $(p,h)$-convex functions such as the Jensen and Hansen-Pedersen type inequalities. Similar to the characterization of operator convex functions given by Hansen and Pedersen \cite{Hansen-Pedersen} we give some equivalent conditions for a function to become an operator $(p,h)$-convex. In applications, we obtain Choi-Davis-Jensen type inequality for operator $(p,h)$-convex functions and a relation between operator $(p,h)$-convex functions with operator monotone functions. 

\section{Class of operator $(p,h)$-convex functions}

Let us now define a new class of operator $(p, h)$-convex functions as follows.
\begin{defn}\label{defn}
Let $h: J \rightarrow \mathbb{R}^+$ be a non-zero super-multiplicative function. A non-negative continuous function $f: K \rightarrow  \mathbb{R}$ is said to be {\it operator $(p,h)$-convex} (or belongs to the class $opgx(p, h, K)$) if 
\begin{equation}\label{def}
f ([\alpha A^p + (1-\alpha)B^p]^{1/p}) \leq h(\alpha)f(A) +h(1-\alpha)f(B).
\end{equation}
for all $A, B \in \mathbb{M}_n^{+}$ with $\sigma(A), \sigma (B) \subset K$, and $\alpha \in (0,1)$.

When $p=1, h(\alpha) = \alpha$ we get the usual definition of operator convex functions on $\mathbb{R}^+$.
\end{defn}  

An operator $(p,h)$-convex function could be either an operator monotone function or an operator convex function. But there are many operator $(p,h)$-convex functions which are neither an operator monotone function nor an operator convex function. Indeed, let $p>0$, $f(t) = t^s$ and $h(\alpha) = \alpha.$ Then the function $f$ is operator $(p, h)$-convex if and only if for any positive definite matrices $A, B$ with spectra in $K$,
$$
(\alpha A^p + (1-\alpha)B^p)^{s/p} \leq h(\alpha) A^s + h(1-\alpha) B^s,
$$
or
$$
(\alpha A + (1-\alpha)B)^{s/p} \leq \alpha A^{s/p} +(1-\alpha) B^{s/p}.
$$
The last inequality means that the function $g(t) = t^{s/p}$ is operator convex, which is equivalent to the condition $s/p \in [1, 2]$. 

In the case $s \in [p, 2p] \cap [0, 1]$ (or $s \in [p, 2p] \cap [1, 2])$, the function $t^s$ is operator monotone (operator convex, respectively) and, at that time, is operator $(p,h)$-convex. If $s$ does not belong to $[p, 2p] \cap [0, 1]$, then the function $t^s$ is operator $(p,h)$-convex but neither operator monotone nor operator convex.

Recall that for arbitrary positive semidefinite matrices $A$ and $B$, the matrix function 
$$
F(p) =\left(\frac{ A^p + B^p}{2}\right)^{1/p}
$$
is called the {\it log Euclidean mean} of $A, B$. In \cite{BS} Bhagwat and Subramanian showed that the matrix function $F(p)$ is monotone with respect to $p$, on the intervals $(-\infty, -1]$ and $[1,\infty)$ but not on $(-1, 1)$.  A more general results about the monotonicity of $F(p)$ was proved by Audenaert and Hiai in \cite{AH}. Now let us prove some properties of operator $(p, h)$-convex functions.
\smallskip

\begin{proposition}
~
\begin{itemize}
\item[(i)] If $f, g \in opgx(p, h, K)$ and $\lambda >0$, then $f + g, \lambda f \in opgx(p, h, K)$;
\item[(ii)] Let $h_1$ and $h_2$ be non-negative and non-zero super-multiplicative functions defined on an interval $J$ with $h_2 \leq h_1$ in $(0,1)$. If $f \in opgx(p, h_2, K)$, then $f \in opgx(p, h_1, K)$;
\item[(iii)] Let $f \in opgx(p_2, h, K)$ and $f$ operator monotone function on $K$. If $1 \le p_1 \le p_2$, then $f \in opgx(p_1, h, K)$. 
\end{itemize}
\end{proposition}

\begin{proof}
(i) The proof immediately follows from the definition of the class $opgx(p, h, K)$.

(ii) Suppose that $f \in opgx(p, h_2, K)$. For any  positive semidefinite matrices $A, B$ with spectra in $K$ and $\alpha \in [0,1]$ we have
\begin{equation*}
\begin{split}
f \big([\alpha A^p + (1-\alpha)B^p]^{1/p}\big) & \leq h_2(\alpha)f(A) +h_2(1-\alpha)f(B) \\
& \leq h_1(\alpha)f(A) +h_1(1-\alpha)f(B).
\end{split}
\end{equation*}
Therefore, $f \in opgx(p, h_1, K)$.

(iii) Put $g(p) = ( \alpha A^p + (1 - \alpha)B^p)^{1/p}$. On account of the above mentioned result \cite{BS}, the function $g(p)$ is monotone increasing on $[1, \infty)$. Since $f$ is operator monotone on $K$ and $f \in opgx(p_1, h, K)$, hence 
\begin{equation*}
f \big( g(p_1) \big) \leq f \big( g(p_2) \big) \leq h(\alpha)f(A) + h(1 - \alpha) f(B).
\end{equation*}
Thus, $f \in opgx(p_1, h, K)$.
\end{proof}

\begin{proposition}\label{pp} Let $K$ be an interval in $\mathbb{R^+}$ such that $0 \in K$. 
\begin{itemize}
\item[(i)] If $f \in opgx(p, h, K), f(0) = 0$, and $h$ is super-multiplicative, then 
\begin{equation} \label{a}
f \big([\alpha A^p + \beta B^p]^{1/p}\big) \leq h(\alpha)f(A) +h(\beta)f(B)
\end{equation}
holds for arbitrary positive semidefinite matrices $A, B$ with spectra in $K$ and all $\alpha, \beta > 0$ such that $ \alpha + \beta \leq 1$.
\item[(ii)] Let $h$ be a non-negative function such that $h(\alpha) < \frac{1}{2}$ for some $\alpha \in (0, \frac{1}{2})$. If $f$ is a non-negative continuous function satisfying (\ref{a}) for all positive definite matrices $A, B$ with spectra in $K$ and all $ \alpha, \beta > 0$ with $\alpha + \beta \leq 1$, then $f(0) = 0$.
\end{itemize}
\end{proposition}

\begin{proof} (i) Let $\alpha, \beta > 0, \alpha + \beta = \gamma < 1$, and let $a$ and $b$ be numbers such that $a = \dfrac{\alpha}{\gamma}$ and $b = \dfrac{\beta}{\gamma}$. Then we have $a + b = 1$ and
\begin{equation*}
\begin{split}
f \big([\alpha A^p + \beta B^p]^{1/p}\big) 
& = f \big([a \gamma A^p + b \gamma B^p]^{1/p}\big) \\
& \leq  h(a)f((\gamma A^p)^{1/p}) +h(b)f((\gamma B^p)^{1/p}) \\
& = h(a)f ((\gamma A^p + (1 - \gamma) O_n^p)^{1/p}) +h(b)f(( \gamma B^p + (1 - \gamma) O_n^p)^{1/p}) \\
& \leq h(a) h(\gamma)f(A) + h(b)h(\gamma)f(B) \\
& \leq h(a \gamma)f(A) + h(b \gamma)f(B) \\
& = h(\alpha)f(A) + h(\beta)f(A)
\end{split}
\end{equation*}

(ii) Suppose that $f(0) > 0$, then we have $f(O_n)=f(0) I_n$. Substitute $ A = B = O_n$ into (\ref{a}), we get
\begin{equation} \label{0}
f(0) I_n= f (( \alpha O_n^p + \beta O_n^p )^{1/p} ) \leq h(\alpha)f(0)I_n + h(\beta) f(0)I_n.
\end{equation}
Let $\alpha = \beta$. Dividing both sides of (\ref{0}) by $f(0)$, we arrive to a contradiction:
 $$2h(\alpha) \geq 1 \quad \hbox{for all} \quad \alpha \in (0, \frac{1}{2}).$$
Thus, $f(0) = 0$.
\end{proof}

\begin{corollary}
Let $h_s(x) = x^s$, where $s, x > 0$, and let     $0 \in K$.  For all $f \in opgx(p, h_s, K)$, the inequality (\ref{a}) holds for all $\alpha, \beta > 0$ with $\alpha + \beta \leq 1$ if and only if $f(0) = 0$. 
\end{corollary}
\begin{proof}
By Proposition \ref{pp}, we just need to consider the case $\alpha, \beta > 0$ with $\alpha + \beta \leq 1$. 

Put $\alpha + \beta = \gamma \leq 1$, and let $a$ and $b$ be positive numbers such that $a = \dfrac{\alpha}{\gamma}$ and $b = \dfrac{\beta}{\gamma}$. Then, $a + b = 1$ and
\begin{equation*}
\begin{split}
f([\alpha A^p + \beta B^p]^{1/p}) & = f([a \gamma A^p + b \gamma B^p]^{1/p}) \\
& \leq h(a) f([\gamma A^p]^{1/p}) + h(b) f([\gamma B^p]^{1/p}) \\
& = a^s f([\gamma A^p]^{1/p}) + b^s f ( [\gamma B^p]^{1/p}) \\
& \leq a^s \gamma^s f(A) + a^s (1 - \gamma)^s f(O_n)  + b^s \gamma^s f(B) + b^s (1 - \gamma)^s f(O_n)  \\
& = a^s \gamma^s  f(A) + b^s \gamma^s  f(B)\\
& = \alpha^s f(A) + \beta^s f(B).
\end{split}
\end{equation*}
Substitute $A = B = O_n, \alpha = \beta = 1/k$ ($k \in \mathbb{N}, k \geq 2$) into (\ref{a}), and then, tend $k$ to the infinite, we get $f(0) \leq 0$. Since $f(0) \geq 0$ by the definition of operator $(p,h)$-convex functions, hence $f(0) = 0$.
\end{proof}

\section{Inequalities for operator $(p,h)$-convex functions}
\subsection{Jensen type inequality for operator $(p,h)$-convex functions}

The classical Jensen inequality for $n$ real numbers states that for any positive tuples $a_i$ and $t_i$ such that $\sum_{i=1}^{n} t_i = 1$
\begin{equation*}
f(\sum_{i=1}^{n}t_ia_i) \leq \sum_{i=1}^n t_if(a_i).
\end{equation*}
The matrix version of Jensen type inequality for operator $(p,h)$-convex functions is as follows.
\begin{theorem} Let $h$ be a non-negative super-multiplicative function on $J$ and $f \in opgx(p, h, K)$. Then for any $k$  positive semidefinite matrices $A_i$ with spectra in $K$ and any $\alpha_i \in (0,1)$ satisfying $\sum_{i = 1}^k {\alpha_i } = 1$, 
\begin{equation} \label{TQ}
f \big( [\sum_{i=1}^k {\alpha_i A_i^p ]^{1/p}} \big) \leq \sum_{i=1}^k {h(\alpha_i) f(A_i)}.
\end{equation}
\end{theorem}
\begin{proof}
We will prove the theorem by the mathematical induction.

When $k = 2$, inequality (\ref{TQ}) reduces to (\ref{def}). 

Assume that (\ref{TQ}) holds for  any $(k -1)$  positive semidefinite matrices with spectra in $K$. We need to prove (\ref{TQ}) for any $k$  positive semidefinite matrices with spectra in $K$. We have
\begin{equation*}
\begin{split}
f( [\sum_{i=1}^k {\alpha_i A_i^p ]^{1/p}}) & = f( [\sum_{i=1}^{k-1} {\alpha_i A_i^p + \alpha_n A_n^p]^{1/p}} \big) \\
& =  f ( [(1- \alpha_n)(\sum_{i=1}^{k-1} {\frac{\alpha_i}{1- \alpha_n} A_i^p) + \alpha_n A_n^p]^{1/p}})\\
& \leq h(1- \alpha_n) f ( [\sum_{i=1}^{k-1} {\frac{\alpha_i }{1- \alpha_n}A_i^p ]^{1/p}}) + h(\alpha_n) f(A_n)\\
& \leq h(1- \alpha_n)\sum_{i=1}^{k-1}{ h(\frac{\alpha_i }{1- \alpha_n})f(A_i)} + h(\alpha_n) f(A_n) \\
& \leq \sum_{i=1}^k {h(\alpha_i) f(A_i)}.
\end{split}
\end{equation*}
The first and the second inequalities follow from the inductive assumption, the third one follows from the super-multiplication of the function $h$.

Thus, (\ref{TQ}) holds for any $k$.
\end{proof}

\begin{remark}
For $h(\alpha) = \alpha$ and $p = 1$, the inequality (\ref{TQ}) reduces to the well-known Jensen inequality for operator convex functions:
\begin{equation*}
f(\sum_{i=1}^k{\alpha_iA_i}) \leq \sum_{i=1}^k{\alpha_if(A_i)}
\end{equation*}
for $\alpha_i \in (0,1)$ and $\sum_{i=1}^k{\alpha_i}=1$.

For $h(\alpha) = \frac{1}{\alpha}$ and $p=1$ we get the Jensen inequality for operator $Q$-class functions:
\begin{equation*}
f(\sum_{i=1}^k{\alpha_iA_i}) \leq \sum_{i=1}^k \frac{f(A_i)}{\alpha_i}
\end{equation*}
for $\alpha_i \in (0,1)$ and $\sum_{i=1}^k{\alpha_i}=1$.

For $h(\alpha) = 1$, $p=1$ we get the Jensen inequality for operator $P$-class functions:
\begin{equation*}
f(\sum_{i=1}^k{\alpha_iA_i}) \leq \sum_{i=1}^k f(A_i)
\end{equation*}
for $\alpha_i \in (0,1)$ and $\sum_{i=1}^k{\alpha_i}=1$.
\end{remark}
\bigskip

As an application of Jensen type inequality we obtain an inequality for index set function. Let $E$ be a finite nonempty set of positive integers and let $F$ be an index set function defined by
\begin{equation}
F(E) = h(W_E)f ([ \frac{1}{W_E} \sum_{i \in E}{w_iA_i^p}]^{1/p} ) - \sum_{i \in E}{h(w_i)f(A_i)},
\end{equation}
where
$
 W_E = \sum_{i \in E}{w_i}.
$
\begin{theorem} \label{ME}
Let $h: \mathbb{R}^+ \rightarrow \mathbb{R}^+$ be a super-multiplicative function, $M$ and $E$ finite nonempty sets of positive integers such that $M \cap E =\emptyset$. Then for any operator $(p,h)$-convex function $f: K \rightarrow \mathbb{R}^+$, for any $w_i > 0$, and for any  positive semidefinite matrices $A_i\ (i \in M \cup E)$ with spectra in $K$, 
\begin{equation} \label{TME}
F(M \cup E) \leq F(M) + F(E).
\end{equation}

\begin{proof}
On account of the operator $(p,h)$-convexity of $f$ and the super-multiplication of $h$, we get 
\begin{equation}\label{set}
\begin{split}
& h(W_{M \cup E}) f ([ \frac{1}{W_{M \cup E}}  \sum_{i \in M \cup E}{w_iA_i^p} ]^{1/p} )\\
 & = h(W_{M \cup E}) f ([ \frac{W_M}{W_{M \cup E}} \sum_{i \in M}{\frac{w_i}{W_M}A_i^p} + \dfrac{W_E}{W_{M \cup E}} \sum_{i \in E}{\frac{w_i}{W_E}A_i^p} ]^{1/p}) \\
& \leq h(W_{M \cup E}) h( \frac{W_M}{W_{M \cup E}}) f([ \sum_{i \in M}{\frac{w_i}{W_M}A_i^p}]^{1/p}) + h(W_{M \cup E}) h( \frac{W_E}{W_{M \cup E}}) f ([ \sum_{i \in E}{\frac{w_i}{W_E}A_i^p}]^{1/p}) \\
& \le h(W_M) f ([\frac{1}{W_M} \sum_{i \in M}{w_iA_i^p}]^{1/p}) + h(W_E) f([\frac{1}{W_E} \sum_{i \in E}{w_i A_i^p}]^{1/p}).\\
\end{split}
\end{equation}
Subtracting from both sides of (\ref{set}) by $\sum_{i \in M \cup E}{h(w_i) f(A_i)}$ and using the identity $\sum_{i \in M \cup E}{h(w_i)f(A_i)} = \sum_{i \in M}{h(w_i)f(A_i)}+\sum_{i \in  E}{h(w_i)f(A_i)}$, we obtain (\ref{ME}).
\end{proof}
\end{theorem}

From Theorem \ref{ME} we get a simple corollary as follow.
\begin{corollary}
Let $h: (0, \infty) \rightarrow \mathbb{R}$ be a non-negative super-multiplicative function. If $w_i > 0\ (i = 1,\cdots, k)$, and $M_l = \{1,\cdots, L\}$, then for $f \in opgx(p, h, K)$, we have
\begin{equation*} \label{coME1}
F(M_k) \leq F(M_{k-1}) \leq ... \leq F(M_2) \leq 0
\end{equation*}
and
\begin{equation*} \label{coME2}
F(M_k) \leq \min_{1 \leq i \leq j \leq k}{\bigg \lbrace h(w_i+w_j)f \bigg( \bigg[\dfrac{w_iA_i^p + w_jA_j^p}{w_i + w_j} \bigg]^{\frac{1}{p}} \bigg) - h(w_i)f(A_i) - h(w_j)f(A_j) \bigg \rbrace}.
\end{equation*}
\end{corollary}

\subsection{Hansen-Pedersen type inequality}
The proof of the following theorem is adapted from the proof of Hansen-Pedersen inequality for operator convex functions \cite{Hansen-Pedersen}.
\begin{theorem}\label{HP}
Let $h: J \rightarrow \mathbb{R}^+$ be a super-multiplicative function, $f: K \rightarrow \mathbb{R}^+$ an operator $(p,h)$-convex function. Then for any pair of positive semidefinite matrices $A$ and $B$ with spectra in $K$ and for matrices $C, D$ such that $CC^* + DD^* =I_n$, 
\begin{equation}
f( [CA^pC^* + DB^pD^*]^{1/p}) \leq 2h(\frac{1}{2}) ( Cf(A)C^* + Df(B)D^*).
\end{equation}
\end{theorem}

\begin{proof} From the condition $CC^* + DD^* = I_n$, it implies that we can find a unitary block matrix 
$$
U: =
  \begin{bmatrix}
 C & D\\
   X & Y\\
  \end{bmatrix}
$$
when the entries $X$ and $Y$ are chosen properly. Then
$$
U \begin{bmatrix}
 A^p & O_n\\
   O_n & B^p\\
  \end{bmatrix}U^* =  \begin{bmatrix}
 CA^pC^* + DB^pD^* & CA^pX^* + DB^pY^* \\
   XA^pC^* + YB^pD^*  &XA^pX^* + YB^pY^* \\
  \end{bmatrix}
$$
It's easy to check that
$$
\dfrac{1}{2}V\begin{bmatrix}
 A_{11} & A_{12}\\
  A_{21} & A_{22}\\
  \end{bmatrix}V + \dfrac{1}{2}\begin{bmatrix}
 A_{11} & A_{12}\\
  A_{21} & A_{22}\\
  \end{bmatrix} = \begin{bmatrix}
 A_{11} & O_n\\
  O_n & A_{22}\\ 
  \end{bmatrix}
$$
for $ V = \begin{bmatrix}
 -I & O_n\\
  O_n & I\\
  \end{bmatrix}.$
It follows that the matrix
$$
Z: = \dfrac{1}{2}VU \begin{bmatrix}
 A^p & O_n\\
  O_n & B^p\\ 
  \end{bmatrix}U^*V + \dfrac{1}{2}U \begin{bmatrix}
 A^p & O_n\\
  O_n & B^p\\ 
  \end{bmatrix}U^* 
$$
is diagonal, where $ \begin{bmatrix}
 A_{11} &A_{12}\\
A_{21}& A_{22}\\ 
  \end{bmatrix} = U\begin{bmatrix}
 A^p & O_n\\
  O_n & B^p\\ 
  \end{bmatrix}U^*.$
It implies  $ Z_{11} = CA^pC^* + DB^pD^*$ and $f(Z_{11}^{1/p}) = f ((CA^pC^* + DB^pD^*)^{1/p})$.
On account of the operator $(p,h)$-convexity of $f$, we have
 \begin{equation*}
  \begin{split}
  f(Z^{1/p}) & = f \bigg(\left(\dfrac{1}{2}VU \begin{bmatrix}
 A^p & O_n\\
  O_n & B^p\\ 
  \end{bmatrix}U^*V + \dfrac{1}{2}U \begin{bmatrix}
 A^p & O_n\\
  O_n & B^p\\ 
  \end{bmatrix}U^*\right)^{1/p} \bigg) \\
  & \leq h(\dfrac{1}{2}) VU f \bigg( \begin{bmatrix}
 A^p & O_n\\
  O_n & B^p\\ 
  \end{bmatrix}^{1/p} \bigg)U^*V + h(\dfrac{1}{2}) U f \bigg( \begin{bmatrix}
 A^p & O_n\\
  O_n & B^p\\ 
  \end{bmatrix}^{1/p} \bigg)U^*  \\
  & = 2h(\dfrac{1}{2}) \bigg( \dfrac{1}{2}VUf \left(\begin{bmatrix}
 A & O_n\\
  O_n & B\\ 
  \end{bmatrix} \right)U^*V + \frac{1}{2}Uf \left(\begin{bmatrix}
 A & O_n\\
  O_n & B\\ 
  \end{bmatrix} \right)U^* \bigg) \\
  & = 2h(\dfrac{1}{2}) \begin{bmatrix}
 Cf(A)C^* + D f(B)D^* & O_n\\
  O_n & Xf(A)X^* + Yf(B)Y^*\\ 
  \end{bmatrix},
    \end{split}
 \end{equation*}
  where 
 $$
 \dfrac{1}{2}VUU^*V + \dfrac{1}{2}UU^* = I_n.
 $$
 Therefore, 
\begin{equation*}
\begin{split}
 f(Z_{11}^{1/p})& = f ( [CA^pC^* + DB^pD^*]^{1/p})\\
 & \leq 2 h(\frac{1}{2})[Cf(A)C^* + Df(B)D^*].
 \end{split}
\end{equation*}
 \end{proof}
    
In the following theorem, we obtain several equivalent conditions for a function to become operator $(p,h)$-convex.
\begin{theorem}\label{apl}
Let $f$ be a non-negative  continuous function on the interval $K$  such that $f(0)=0$, and $h$ a non-negative and non-zero super-multiplicative function on $J$ satisfying $2h(1/2) \le \alpha^{-1} h(\alpha)\ (\alpha \in (0,1))$. Then the following statements are equivalent: 
\begin{itemize}
\item[(i)] $f$ is an operator $(p,h)$-convex function;

\item[(ii)] for any contraction $||V|| \leq 1$ and self-adjoint matrix $A$ with spectrum in $K$,
$$f ((V^*A^pV)^{1/p}) \leq 2h(\frac{1}{2})V^*f(A)V;$$

\item[(iii)] for any orthogonal projection $Q$ and any  positive semidefinite matrix $A$ with $\sigma(A) \subset K$,
$$f((QA^pQ)^{1/p}) \leq 2h(\frac{1}{2}) Qf(A)Q;$$

\item[(iv)] for any natural number $k$, for any families of positive operators $\{A_i\}_{i=1}^k$ in a finite dimensional Hilbert space $H$ satisfying $\sum_{i=1}^k \alpha_i A_i = I_H$ (the identity operator in $H$) and for arbitrary numbers $x_i \in K$,
\begin{equation}\label{2}
f ([ \sum_{i =1}^k{\alpha_i x_i^p A_i}]^{1/p}) \leq \sum_{i =1}^k {h(\alpha_i) f(x_i)A_i}.
\end{equation}
\end{itemize}
\end{theorem}

\begin{proof}
The implication (ii) $\Rightarrow$ (iii) is obvious. 

Let us prove the implication (i) $\Rightarrow$ (ii). Suppose that $f \in opgx(p, h, J)$. Then by Theorem \ref{HP} we have 
$$
f([CA^pC^* + DB^pD^*]^{1/p}) \leq 2h(\frac{1}{2}) ( Cf(A)C^* + Df(B)D^*),
$$
where $ CC^* + DD^* =I_n.$ Since $||V|| \leq 1$, we can choose $W $ such that $VV^* +WW^* = I_n$. Choosing $B = O_n$, we have that $f(B) =f(O_n)=  f(0) I_n= O_n$. Hence,  
\begin{equation*}
\begin{split}
f ( (V^*A^pV)^{1/p})& = f ( (V^*A^pV + W^*B^pW)^{1/p})  \\
& \leq 2h(\frac{1}{2}) ( V^*f(A)V + W^*f(B)W)  \\
& \leq 2h(\frac{1}{2}) ( V^*f(A)V).     \\
\end{split}
\end{equation*}

(iii) $\Rightarrow$ (i). Let $A$ and $B$ be  positive semidefinite matrices with spectra in $K$ and $0< \lambda <1$.
Define
$$
C : = \begin{bmatrix}
 A & O_n\\
  O_n & B\\ 
  \end{bmatrix}, ~~ U := \begin{bmatrix}
 \sqrt{\lambda}I_n & -\sqrt{1- \lambda}I_n\\
  \sqrt{1- \lambda}I_n &  \sqrt{\lambda}I_n\\ 
  \end{bmatrix},  ~~ Q :=  \begin{bmatrix}
 I_n & O_n\\
  O_n & O_n\\ 
  \end{bmatrix}.
$$ 
Then $C = C^*$ with $\sigma(C) \subset K$, $U$ is an unitary and $Q$ is an orthogonal projection and 
\begin{equation*}
U^*C^pU =  \begin{bmatrix}
 \lambda A^p + (1- \lambda)B^p & - \sqrt{\lambda - \lambda^2}A^p +\sqrt{\lambda - \lambda^2}B^p\\
   - \sqrt{\lambda - \lambda^2}A^p +\sqrt{\lambda - \lambda^2}B^p &  (1- \lambda)A^p +\lambda  B^p\\ 
  \end{bmatrix}
\end{equation*}
is a self-adjoint matrix. Since
$$
QU^*C^pUQ =  \begin{bmatrix}
 \lambda A^p + (1- \lambda)B^p & O_n\\
  O_n & O_n\\ 
  \end{bmatrix}
$$ 
and $||UQ|| \leq 1$, hence
\begin{equation*}
\begin{split}
 \begin{bmatrix}
 f ( (\lambda A^p + (1- \lambda)B^p)^{1/p}) & O_n\\
  O_n & O_n\\ 
  \end{bmatrix}& = f \big( (QU^*C^pUQ)^{1/p} \big) \\
  & \leq 2h(\frac{1}{2}) QU^*f(C)UQ  \\ 
  & = 2h(\frac{1}{2})\begin{bmatrix}
 \lambda f(A) + (1- \lambda)f(B) & O_n\\
  O_n & O_n\\ 
  \end{bmatrix}.
  \end{split}
\end{equation*}
According to the property of $h$, from the last inequality it implies
\begin{equation*} \label{DK}
\begin{split}
 f (( \lambda A^p + (1- \lambda)B^p)^{1/p}) & \leq 2h(\dfrac{1}{2}) (\lambda f(A) + (1- \lambda)f(B))\\
 & \le h(\lambda) f(A) + h(1-\lambda)f(B).
 \end{split}
\end{equation*}

(iv) $\Rightarrow$ (i). Let $X, Y$ be two  positive operators on $H$ with spectra in $K$, and $\alpha \in (0,1)$. Let  
$$X= \sum_{i=1}^n{\lambda_iP_i}, \quad Y =\sum_{j=1}^n{\mu_jQ_j}$$ 
be the spectral decompositions of $X$ and $Y$. Then we have
$$
\alpha \sum_{i=1}^n{P_i} + (1- \alpha)\sum_{j=1}^n{Q_j}= I_H.
$$
On account of (\ref{2}), we have
\begin{equation*}
\begin{split}
f ([\alpha X^p + (1- \alpha) Y^p]^{1/p}) &= f ([ \sum_{i=1}^n {\alpha \lambda_i^p P_i} + \sum_{j=1}^n {(1 - \alpha) \mu_i^p  Q_j}]^{1/p}) \\
& \leq  \sum_{i=1}^n {h( \alpha) f( \lambda_i) P_i} + \sum_{j=1}^n {h(1- \alpha) f( \mu_j) Q_j } \\
& = h(\alpha) \sum_{i=1}^n f (\lambda_i) P_i + h(1- \alpha) \sum_{j=1}^n {f(\mu_j) Q_j}\\
& = h (\alpha)f(X) + h(1- \alpha) f(Y).
\end{split}
\end{equation*}

(i) $\Rightarrow$ (iv). By the Neumark theorem \cite{Neumark}, there exists a Hilbert space $\mathcal{H}$ larger than $H$ and a family of mutually orthogonal projections $P_i$ in $\mathcal{H}$ such that $\sum_{i=1}^k {P_i} = I_\mathcal{H}$ and $\alpha_i A_i = P P_i P|_H (i =1, 2,..., k)$, where $P$ is the projection from $\mathcal{H}$ onto $H$. Then we have
\begin{equation*}
\begin{split}
f ([ \sum_{i=1}^k{\alpha_i x_i^p A_i }]^{1/p}) & = f([ \sum_{i=1}^k{ x_i^p P P_i P |_ H}]^{1/p}) \\
& =  f([P (\sum_{i=1}^k{ x_i^p  P_i })P |_ H]^{1/p}) \\
& \leq  2 h( \frac{1}{2}) P  f ( [\sum_{i=1}^k{ x_i^p  P_i } ]^{1/p})  P |_ H  \\
& = 2 h( \frac{1}{2}) P( \sum_{i=1}^k { f(x_i) P_i } ) P |_ H \\
& = 2 h( \frac{1}{2})  \sum_{i=1}^k { f(x_i) P P_i P |_ H} \\
& = 2 h( \frac{1}{2})  \sum_{i=1}^k {\alpha_i f(x_i) A_i} \\
& \leq  \sum_{i=1}^k {h( \alpha_i) f(x_i) A_i}  
\end{split}
\end{equation*}
\end{proof}

As a consequence of the above theorem we obtain the Choi-Davis-Jensen type inequality for operator $(p, h)$-convex functions.
\begin{corollary}\label{cor}
Let $\Phi$ be a unital positive linear map on $B(H)$, $A$ a positive operator in $H$ and $f$ an operator $(p, h)$-convex function on $\mathbb{R}^+$ such that $f(0) = 0$. Let $h$ be a non-negative and non-zero super-multiplicative function on $J$ satisfying $2h(1/2) \le \alpha^{-1} h(\alpha)\ (\alpha \in (0,1))$. Then
$$
f((\Phi(A^p))^{1/p}) \le 2h(1/2) \Phi(f(A)).
$$
\end{corollary}

\begin{proof}
Let $A$  be an arbitrary positive operator in $H$. We put $\Psi$ the restriction of $\Phi$ to the $C^*$-algebra $\mathcal{C}^*(A, I_H)$ generated by $I_H$ and $A$. Then $\Psi$ is a unital completely positive map on $\mathcal{C}^*(A, I_H)$. By the Stinespring dilation theorem, there exists an isometry $V: H \mapsto H$ and a unital $*$-homomorphism $\pi: \mathcal{C}^*(A, I_H) \mapsto B(H)$ such that $\Psi(A) = V^*\pi(A) V$. Hence,
\begin{equation*}
\begin{split}
f((\Phi(A^p))^{1/p}) & = f((\Psi(A^p))^{1/p}) = f((V^*\pi(A)^p V)^{1/p}) \le 2h(1/2) V^*f(\pi(A)) V \\
& = 2h(1/2) V^*\pi(f(A)) V = 2h(1/2) \Psi(f(A)) = 2h(1/2) \Phi(f(A)).
\end{split}
\end{equation*}
\end{proof}

\begin{remark} Here we give an example for the function $h$ which is different from the identity function and satisfies conditions in Theorem \ref{apl} and Corollary \ref{cor}. It is easy to check that for the function $h(x) = x^3 -x^2 +x$ and for any $x, y \in [0, 1]$,
$$
h(xy) -h(x)h(y) = xy(x+y)(1-x)(1-y) \ge 0.
$$
Therefore, $h$ is super-multiplicative on $[0, 1]$. At the same time,  the function $h(x)/x = x^2 - x +1$ attains minimum at $x =1/2$, and hence $2h(1/2) \le  h(x) /x$ for any $x \in (0, 1)$. 
\end{remark}

\begin{corollary}\label{10}
Let $f$ be operator $(1,h)$-convex function on $\mathbb{R}^+$ such that $f(0) = 0$. Then for any positive definite matrices $A \le B$, 
$$
A^{-1} f(A) \le 2h(1/2)B^{-1} f(B).
$$

In the case when $2h(1/2) \le 1$ the function $t^{-1} f(t)$ is operator monotone on $(0, \infty)$, and hence the function $f(t)$ is operator convex.
\end{corollary}
\begin{proof}
Since $0<A \le B,$ we can find $C$ such that $A^{1/2} = CB^{1/2}$, and hence $A= CBC^*$. Then
\begin{align*}
A^{-1} f(A) &= B^{-1/2}C^{-1}f(C BC^*)(C^*)^{-1}B^{-1/2} \\
& \le 2h(1/2) B^{-1/2}C^{-1}Cf(B) C^*(C^*)^{-1}B^{-1/2} \\
&= 2h(1/2) B^{-1} f(B).
\end{align*}

In the case when $2h(1/2) \le 1$, from the above inequality we get
$$
A^{-1} f(A)  \le B^{-1} f(B), 
$$
that means, the function $t^{-1}f(t)$ is operator monotone, and as a consequence of that  the function $f(t)$ is operator convex by \cite{Hansen-Pedersen}.
\end{proof}

\begin{remark}
It is easy to check that the function $h(x) = (x^3-x^2+x)/2$ is super-multiplicative and satisfies the conditions in Theorem \ref{apl} and Corollary \ref{10}.
\end{remark}

\subsection{An open question}

As we know that the value $f(p) = (\frac{a^p+b^p}{2})^{1/p}$ is called the {\it binormal mean}, or the {\it power mean}, and is an increasing function of $p$ on $(-\infty, \infty)$. And it is well-known that for two positive number $a, b$,
$$
\sqrt{ab} = e^{1/2(\log a + \log b)} = \lim_{p \mapsto 0} \left(\frac{a^p+b^p}{2}\right)^{1/p}.
$$

In the other hand, Bhagwat and Subramanian \cite{BS} showed that for positive definite matrices $A, B$,
$$
\lim_{p\mapsto 0} F(p) =  e^{\frac{1}{2}(\log A+ \log B)} 
$$
and $e^{\frac{1}{2}(\log A+\log B)}$ is different to the geometric mean of $A, B$. 

Now, suppose that $h(\alpha) = \alpha$ and the function $f$ is operator $(p,h)$-convex for any $p>0$ and the function $f$ is continuous on $\mathbb{R}^+$. So, we have
$$
f(e^{\frac{1}{2}(\log A+\log B)}) \le \frac{f(A)+ f(B)}{2}.
$$

{\it Question: what is the class of functions satisfying the last inequality for any positive definite matrices $A, B$?}

\bigskip

{\it Acknowledgement.} The authors would like to express sincere thanks to the anonymous referee for his comments which improve this paper.

\end{document}